\documentclass[reqno,a4paper,12pt]{amsart}

\usepackage{amsmath, amsthm, amssymb}
\usepackage[hidelinks]{hyperref}
\usepackage{enumerate}
\usepackage{url}

\usepackage{bm}

\usepackage{xcolor}

\usepackage[centering, margin={1in, 0.5in}, includeheadfoot]{geometry}

\usepackage{fancyvrb}

\newtheorem{theorem}{Theorem}
\newtheorem{lemma}[theorem]{Lemma}

\theoremstyle{definition}
\newtheorem*{defn}{Definition}

\theoremstyle{remark}

\newtheorem*{remark*}{Remark}

\def\d{{\,{\rm d}}}

\def \ba {\mathbf a}

\def \N {\mathbb{N}}

\def \R {\mathbb{R}}

\def \Z {\mathbb{Z}}

\def\eps{\varepsilon}
\def\supp{\mathrm{supp}}

\title{Sets of nice recurrence are partition regular}
\author{Jonathan Chapman}
\address{Mathematics Institute, Zeeman Building, University of Warwick, Coventry CV4 7AL, United Kingdom}
\email{Jonathan.Chapman@warwick.ac.uk}


\begin{document}

\begin{abstract}
A set of positive integers $R$ is called a set of nice recurrence if for any measure preserving system $(X,\mu,T)$, for all measurable $A\subseteq X$, and each $\varepsilon>0$, there exists $n\in R$ such that $\mu(A\cap T^{-n}A) \geqslant \mu(A)^2 - \varepsilon$. Answering a long-standing question of Bergelson, we show that sets of nice recurrence have the following Ramsey property: any finite colouring of a set of nice recurrence admits a monochromatic set of nice recurrence.
\end{abstract}

\maketitle

\section{Introduction}

A \emph{measure preserving system} $(X,\mu,T)$ consists of a probability measure space $(X,\mu)$ and a measurable function $T:X\to X$ satisfying $\mu(T^{-1}A) = \mu(A)$ for all measurable $A\subseteq X$. A set of positive integers $R$ is called a \emph{set of recurrence} if for all measure preserving systems $(X,\mu,T)$, each $\eps>0$, and all measurable $A\subseteq X$ with $\mu(A)>0$, there exists $n\in R$ such that $\mu(A\cap T^{-n}A)>0$. The fact that the set of all positive integers $\N$ is a set of recurrence, and hence that sets of recurrence exist, is an immediate consequence of Poincar\'e's recurrence theorem.

Sets of recurrence were introduced by Furstenberg \cite{Fur1981} as part of his seminal work on applying dynamical methods to combinatorics and number theory. Furstenberg's correspondence principle \cite[Theorem 1.1]{Fur1977}, which is the central tool of Ergodic Ramsey Theory, allows one to establish results on arithmetic configurations in sets with positive upper density from recurrence properties of positive measure sets in measure preserving systems. Here, the \emph{upper density} of a set of positive integers $S$ is
\begin{equation*}
    \overline{d}(S) :=\limsup_{N\to\infty}\frac{|S\cap\{1,\ldots,N\}|}{N}.
\end{equation*}
For example, Furstenberg \cite[Proposition 1.3]{Fur1977} showed that the squares are a set of recurrence, and combined this with his correspondence principle to show that any set of positive integers with positive upper density contains two numbers which differ by a square.

Over the subsequent decades, there has been substantial research into classifying various notions of recurrence and establishing further combinatorial consequences \cite{KMF1978,Ber1986,For1990,McC1995,FLW2006,BL2008,Gri2019,Ack2022,DLMS2023,FS2025,Rod2025,Zel2026}. In this paper, we will be concerned with a strong notion of recurrence introduced by Bergelson \cite[Definition 2.2]{Ber1986} called \emph{nice recurrence} (also known as \emph{optimal recurrence}).
\begin{defn}[Nice recurrence]
    A set of positive integers $R$ is called a \emph{set of nice recurrence} if for all measure preserving systems $(X,\mu,T)$, all measurable $A\subseteq X$, and all $\eps>0$, there exists $n\in R$ such that
    \begin{equation*}
        \mu(A\cap T^{-n}A) \geqslant \mu(A)^2 - \eps.
    \end{equation*}
\end{defn}

In the same way that Poincar\'e's recurrence theorem implies $\N$ is a set of recurrence, Khintchine's recurrence theorem shows that $\N$ is a set of nice recurrence. More generally, Bergelson observed that sets of positive integers $S$ with upper density $1$ have the following \emph{Ramsey property}: for any finite colouring $S=C_1\cup\cdots\cup C_r$, one of the colour classes $C_i$ must be a set of nice recurrence. Combining this with Furstenberg's correspondence principle, Bergelson \cite[Theorem 1.1]{Ber1986} showed that if the positive integers are finitely coloured $\N=C_1\cup\cdots\cup C_r$, then there exists a colour class $C_i$ with $\overline{d}(C_i)>0$ such that $\overline{d}(\{n\in C_i:\overline{d}(C_i\cap (C_i -n)) \geqslant \overline{d}(C_i)^2 - \eps\}) > 0$ holds for all $\eps>0$. This is a substantial strengthening of a classical theorem of Schur, which only guarantees that there exists a colour class $C_i$ with $C_i\cap (C_i - C_i)\neq\emptyset$.

Motivated by this argument, Bergelson has frequently asked \cite{Ber1986,Ber1996,For1990,BL2008,Kra2024} whether all sets of nice recurrence have the Ramsey property, meaning that if $R=R_1\cup\cdots\cup R_r$ is a set of nice recurrence, then one of the $R_i$ must also be a set of nice recurrence. In combinatorial terminology, this asks whether the family of sets of nice recurrence is \emph{partition regular}. The purpose of this note is to provide a positive answer to this question.

\begin{theorem}[Sets of nice recurrence are partition regular]\label{thm1}
    Let $R$ be a set of positive integers. If $R$ is a set of nice recurrence and $R=R_1\cup\cdots \cup R_r$, then there exists $1\leqslant i\leqslant r$ such that $R_i$ is a set of nice recurrence.
\end{theorem} 

\subsection*{Notation and terminology}

A typical measure preserving system will be denoted by $(X,\mu,T)$ or $(Y,\nu,S)$, often with additional subscripts. For a non-negative integer $n$, we let $T^n$ denote the $n$-fold composition of $T$. Explicitly, we write $T^0(x):=x$ and $T^{n+1}(x):=T(T^n(x))$ for all $n\geqslant 0$. Following convention, given a function $f$ defined on $X$, we write $Tf := f\circ T$. We use the preimage notation $T^{-n}A:=(T^n)^{-1}(A)=\{x\in X: T^n(x)\in A\}$. We also let $\cdot$ denote pointwise multiplication of functions. For example, the expression
\begin{equation*}
    \int_X f\cdot T^nf\d\mu
\end{equation*}
is the integral of the function $(f\cdot T^nf)(x):=f(x)f(T^n(x))$.

The product $(X\times Y, \mu\otimes\nu)$ of two probability measure spaces $(X,\mu)$ and $(Y,\nu)$ is the probability space whose $\sigma$-algebra of measurable sets is generated by all sets of the form $A\times B$ with $A\subseteq X$ and $B\subseteq Y$ measurable and whose measure $\mu\otimes\nu$ satisfies $\mu\otimes \nu(A\times B) = \mu(A)\nu(B)$ (see \cite[\S1.7]{Tao2011} for further details). The corresponding product $(X\times Y, \mu\otimes\nu, T\times S)$ of the measure preserving systems $(X,\mu,T)$ and $(Y,\nu,S)$ additionally has the measure preserving map $(T\times S)(x,y):= (T(x),S(y))$. The product of $i$ copies of $(X,\mu,T)$ and $j$ copies of $(Y,\nu,S)$ is written as $(X^i\times Y^j,\mu^{\otimes i}\otimes\nu^{\otimes j}, T^{\times i}\times S^{\times j})$, where $T^{\times i}:X^i\to X^i$ is the map $T^{\times i}(x_1,\ldots,x_i):= (T(x_1),\ldots,T(x_i))$, and similarly for $S^{\times j}$. If $i=0$, then this product system is understood to be $(Y^j,\nu^{\otimes j},S^{\times j})$, and similarly when $j=0$. We never consider the case where $i=j=0$.


\subsection*{Acknowledgements} We thank Joel Moreira and Rigoberto Zelada for helpful discussions. JC is supported by EPSRC through Joel Moreira's Frontier Research Guarantee grant, ref. \texttt{EP/Y014030/1}.

\subsection*{Rights}

For the purpose of open access, the author has applied a Creative Commons Attribution (CC-BY) licence to any Author Accepted Manuscript version arising from this submission.

\section{Proof of Theorem \ref{thm1}}

We begin with a functional reformulation of nice recurrence. Rather than working only with sets with positive measure, this allows us to use arbitrary measurable `balanced' functions, i.e. bounded functions with mean zero. A similar functional reduction via `Bernoulli extensions' was used in \cite[Proof of Theorem 10.1]{BTZ2015} to investigate multiple recurrence for ergodic $\mathbb{F}_p^\omega$-actions.

\begin{lemma}\label{lem1}
    A set of positive integers $R$ is not a set of nice recurrence if and only if there exists a measure preserving system $(X,\mu,T)$ and a bounded measurable function $f=f_R:X\to\R$ such that
    \begin{equation}\label{eqn1}
        \int_X f \d\mu = 0, \qquad\text{and}\qquad \sup_{n\in R}\int_X f\cdot T^nf\d\mu < 0.\tag{$\dagger$}
    \end{equation}
\end{lemma}
\begin{proof}
    We start with the ``only if'' direction. If $R$ is not a set of nice recurrence, then there is a measure preserving system $(X,\mu,T)$, some $\eps > 0$, and a measurable set $A\subseteq X$ with $\sup_{n\in R}\mu(A\cap T^{-n} A) \leqslant \mu(A)^2 - \eps$. Choosing $f = 1_A - \mu(A)$, we see that $f$ is a bounded mean zero function and, for all $n\in R$, satisfies
    \begin{align*}
        \int_X f\cdot T^nf \d\mu &= \int_X 1_A\cdot T^n 1_A \d\mu - \mu(A)\left( \int_X 1_A\d\mu + \int_X T^n1_A\d\mu - \mu(A)\right)\\
        &= \mu(A\cap T^{-n}A) - \mu(A)^2 \leqslant -\eps < 0.
    \end{align*}

    Now we consider the ``if'' direction. Suppose that $(X,\mu,T)$ and $f:X\to\R$ are as in the statement of the lemma. By dividing $f$ by a sufficiently large positive constant, we may assume that $f$ takes values in $[-1,1]$. Set
    \begin{equation*}
        \eps = -\frac{1}{4}\sup_{m\in R}\int_X f\cdot T^m f\d\mu > 0.
    \end{equation*}
    Now let $(Y,\nu,S)$ be the product of the measure preserving systems $(\Omega,\lambda^{\otimes\Z},\sigma)$ and $(X,\mu,T)$, where $\Omega = [-1,1]^{\Z}$ has the product topology and Borel $\sigma$-algebra, $\lambda$ is the normalised Lebesgue measure, and $\sigma$ is the shift map $(\sigma(\ba))_n = a_{n+1}$. Let
    \begin{equation*}
       A = \{(\ba,x)\in\Omega\times X: a_0\leqslant f(x)\}.
    \end{equation*}
        Since $f$ takes values in $[-1,1]$ and has mean zero, an application of the Fubini-Tonelli theorem (see \cite[Theorem 1.7.15]{Tao2011}) gives
    \begin{equation*}
        \nu(A) = \int_X \lambda^{\otimes\Z}(\{\ba\in\Omega:a_0\leqslant f(x)\})\d\mu(x) = \frac{1}{2}\int_X(1+f)\d\mu = \frac{1}{2}.
    \end{equation*}
    Similarly, upon noting that $T^{-n}(A) = \{(\ba,x): a_n\leqslant f(T^n x)\}$, for each $n\in\N$ we have
    \begin{align*}
        4\nu(A\cap S^{-n}A) &= 4\int_X \lambda^{\otimes\Z}(\{\ba\in\Omega:a_0\leqslant f(x),\; a_n\leqslant f(T^n x)\})\d\mu(x)\\
        &= 1 + \int_X f\cdot T^n f\d\mu + \int_X f \d\mu + \int_X T^n f\d \mu\\
        &= 1 + \int_X f\cdot T^n f \d\mu.
    \end{align*}
    Thus, if $n\in R$, then
    \begin{equation*}
        \nu(A\cap S^{-n}(A)) \leqslant \frac{1}{4}\left(1 + \sup_{m\in R}\int_X f\cdot T^m f\d\mu \right)= \nu(A)^2 - \eps,
    \end{equation*}
    which shows that $R$ is not a set of nice recurrence.
\end{proof}

\begin{remark*}
As in \cite{BTZ2015}, Lemma 2 can readily be adapted to characterise sets of nice recurrence for actions of arbitrary abelian groups $G$ by replacing $\Omega$ with  $[-1,1]^G$. Moreover, although we will not make use of this observation, by taking $\Omega=[\frac{-\alpha}{1-\alpha},1]^\Z$ and scaling $f$ so that $\lVert f\rVert_\infty\leqslant \min\{1,\frac{\alpha}{1-\alpha}\}$, the proof of Lemma \ref{lem1} also shows that if $R$ is not a set of nice recurrence, then for any $0<\alpha<1$ one can find a measure preserving system $(X,\mu,T)$, some $\eps>0$, and $A\subseteq X$ with $\mu(A)=\alpha $ such that $\mu(A\cap T^{-n}A)<\mu(A)^2 - \eps$ for all $n\in R$. For $\alpha = 1/2$, this latter fact - alongside a similar functional characterisation of nice recurrence for countable discrete abelian groups - was recently established by Zelada \cite[Appendix B]{Zel2026} using different, more technical methods.
\end{remark*}

To prove Theorem \ref{thm1}, we first apply Lemma \ref{lem1} to obtain functions $f_i:X_i\to\R$ for $i\in\{1,2\}$ which witness the fact that $R_1,R_2\subseteq\N$ are not sets of nice recurrence. We then use these functions to construct a measure preserving system $(X,\mu,T)$ and a function $f:X\to \R$ which satisfies \eqref{eqn1} for all $n\in R_1\cup R_2$. 
The most straightforward approach one could try to accomplish this would be to take $(X,\mu,T)$ to be the product of the $(X_i,\mu_i,T_i)$ and set $f(x_1,x_2):=f_1(x_1)f_2(x_2)$. We would then have
\begin{equation*}
    \int_X f\cdot T^nf \d\mu = \left(\int_{X_1}f_1\cdot T_1^nf_1 \d\mu_1\right)\left(\int_{X_2}f_2\cdot T_2^nf_2 \d\mu_2\right).
\end{equation*}
This is the standard method which is used to show that sets of recurrence are partition regular; see, for example, \cite[Proposition 1.3]{For1990}. Furthermore, one can adapt this argument to show that a set of nice recurrence cannot be written as a union of two sets which are both not sets of $3$-nice recurrence (see \cite[Proposition 2.3]{For1990}). 
However, this is not enough to prove Theorem \ref{thm1}. For example, even though the integral over $X_1$ on the right-hand side is negative when $n\in R_1$, the integral over $X_2$ may also be negative whenever $n\in R_1$. Returning to the original definition of nice recurrence, this issue corresponds to the possibility that, when $n\in R_1$, the measure $\mu_2(A_2\cap T_2^{-n}A_2)$ is significantly larger than $\mu_2(A_2)^2$, which could then prevent $\mu(A\cap T^{-n}A)=\mu_1(A_1\cap T_1^{-n}A_1)\mu_2(A_2\cap T_2^{-n}A_2)$ from being smaller than $\mu(A)^2 = \mu_1(A_1)^2\mu_2(A_2)^2$.

To overcome this issue, we consider a much larger product system for which $X$ takes the form $X=\prod_{(i,j)\in I} (X^i\times X^j)$, for some finite set of indices $I$, and then construct our function $f:X\to\R$ so that
\begin{equation*}
    \int_X f\cdot T^nf \d\mu = P\left(\int_{X_1}f_1\cdot T_1^nf_1 \d\mu_1,\int_{X_2}f_2\cdot T_2^nf_2 \d\mu_2\right)
\end{equation*}
for some polynomial $P\in\R[U,V]$. If we can find such a polynomial $P$ with the property that $P(u,v)$ is negative and bounded away from zero whenever one of $u$ or $v$ is at most $-\eps$, then this would show that the integral on the left-hand side of the above expression is negative and bounded away from $0$ for all $n\in R_1\cup R_2$, thereby proving Theorem \ref{thm1}. The construction of such a polynomial $P$ is accomplished in the following lemma.

\begin{lemma}\label{lem2}
    For all $0<\eps \leqslant 1$, there exists $\delta>0$ and a polynomial $P\in\R[U,V]$ of the form $P(U,V)=\sum_{i,j\geqslant 0}a_{i,j}U^iV^j$ with $a_{i,j}\geqslant 0$ for all $i,j\geqslant 0$ such that $P(0,0) = a_{0,0} = 0$ and $P(u,v)\leqslant - \delta$ for all $u,v\in[-1,1]$ with $\min\{u,v\}\leqslant -\eps$.
\end{lemma}
\begin{proof}
    Since $0<\eps\leqslant 1$, we can find $d\in\N$ such that $(1-\eps)^d <  d\eps $. Now set
    \begin{equation*}
        P(U,V) = U(1+V)^d + V(1+U)^d + \frac{\eps}{2}UV(U+V).
    \end{equation*}
It is immediate that $P(0,0) = 0$ and all the coefficients of all the monomials appearing in $P$ are non-negative. Since $P(U,V)=P(V,U)$, it only remains to show that $P(u,v)<0$ for all $u\in[-1,-\eps]$ and $v\in[-1,1]$. Indeed, once this is established, continuity of $P$ and compactness of $[-1,-\eps]\times[-1,1]$ allow us to finish the proof by setting
    \begin{equation*}
         \delta = - \max\{P(u,v):u,v\in[-1,1],\; \min\{u,v\}\leqslant-\eps\} > 0.
    \end{equation*}
    
    If $u\leqslant -\eps$ and $v \leqslant 0$, then, as each of the three terms in the definition of $P$ are non-positive and cannot all vanish, we see that $P(u,v)<0$. If instead $u\leqslant -\eps < 0 < v$, then an application of Bernoulli's inequality and the trivial bound $uv(u+v)\leqslant 2$ gives
    \begin{align*}
        P(u,v)\leqslant -\eps (1+dv) + v(1-\eps )^d + \eps = v ((1-\eps)^d - d\eps).
    \end{align*}
    Our choice of $d$ and the assumption $v>0$ therefore shows that the right-hand side is negative, completing the proof.
\end{proof}

\begin{proof}[Proof of Theorem \ref{thm1}]
    By contraposition and induction on the number of colours, it suffices to show that if $R_1,R_2\subseteq\N$ are not sets of nice recurrence, then neither is $R_1\cup R_2$. Since $R_1$ and $R_2$ are not sets of nice recurrence, for each $i\in\{1,2\}$, Lemma \ref{lem1} supplies us with measure preserving systems $(X_i,\mu_i,T_i)$ and functions $f_i:X_i\to\R$ such that 
    \begin{equation*}
        \int_{X_i} f_i \d\mu_i = 0, \qquad\text{and}\qquad \sup_{m\in R_i}\int_{X_i} f_i\cdot T_i^m f_i \d\mu_i < 0.
    \end{equation*}
    As in the proof of Lemma \ref{lem1}, we may assume each of the $f_i$ take values in $[-1,1]$. Set
    \begin{equation*}
        \eps = -\max_{i\in\{1,2\}}\left(\sup_{m\in R_i}\int_{X_i} f_i\cdot T_i^m f_i \d\mu_i\right),
    \end{equation*}
    and note that $0<\eps\leqslant 1$. Let $P\in\R[U,V]$ and $\delta > 0$ be as provided by Lemma \ref{lem2} with this choice of $\eps$. Writing
    \begin{equation*}
        P(U,V) = \sum_{i,j\geqslant 0}a_{i,j}U^iV^j,
    \end{equation*}
    let $\supp(P) = \{(i,j)\in(\N\cup\{0\})^2 : a_{i,j}\neq 0\}$. For each $(i,j)\in\supp(P)$, define the measure preserving system $(X_{i,j},\mu_{i,j},T_{i,j}) = (X_1^{i}\times X_2^j,\mu_1^{\otimes i}\otimes\mu_2^{\otimes j}, T_{1}^{\times i}\times T_2^{\times j})$, and let $f_{i,j}:X_{i,j}\to\R$ be the function
    \begin{equation*}
        f_{i,j}(x_1^{(1)},\ldots,x_i^{(1)};x_1^{(2)},\ldots,x_j^{(2)}) = f_1(x_1^{(1)})\cdots f_1(x_i^{(1)})f_2(x_1^{(2)})\cdots f_2(x_j^{(2)}).
    \end{equation*}
    Taking the product of all these systems allows us to construct the measure preserving system $(X,\mu,T)$, where
    \begin{equation*}
       X= \prod_{(i,j)\in\supp(P)} (X_1^i\times X_2^j);\quad \mu = \bigotimes_{(i,j)\in\supp(P)} (\mu_1^{\otimes i}\otimes\mu_2^{\otimes j}); \quad T =\prod_{(i,j)\in\supp(P)} (T_1^{\times i}\times T_2^{\times j}).
    \end{equation*}
    Finally, we define the function $f:X\to\R$ by
    \begin{align*}
        f\left( \prod_{(i,j)\in\supp(P)} (x_1^{(1)},\ldots,x_i^{(1)};x_1^{(2)},\ldots,x_j^{(2)})_{(i,j)}\right) \\= \sum_{(i,j)\in\supp(P)}a_{i,j}^{1/2}f_{i,j}\left((x_1^{(1)},\ldots,x_i^{(1)};x_1^{(2)},\ldots,x_j^{(2)})_{(i,j)}\right),
    \end{align*}
    where the $(i,j)$ subscripts on $(x_1^{(1)},\ldots,x_i^{(1)};x_1^{(2)},\ldots,x_j^{(2)})_{(i,j)}$ indicates that the tuple is an element of $X_1^{\times i}\times X_2^{\times j}$ and is independent of the corresponding tuple for $X_1^{\times i'}\times X_2^{\times j'}$ for any $(i',j')\neq (i,j)$.

    To complete the proof, we show that the function $f$ we have constructed satisfies \eqref{eqn1}. Recall from Lemma \ref{lem2} that if $(i,j)\in\supp(P)$, then $\max\{i,j\}\geqslant 1$. Hence, as $f_1$ and $f_2$ satisfy \eqref{eqn1}, for each $(i,j)\in\supp(P)$, we have\footnote{Recall that $i=0$ (respectively $j=0$) corresponds to the integral over $X_1$ (respectively $X_2$) being omitted. Hence, these expressions remain valid if one uses the convention $0^0 = 0$.}
    \begin{equation*}
        \int_{X_{i,j}}f_{i,j}\d\mu_{i,j} = \left(\int_{X_1}f_1 \d\mu_1\right)^i\left(\int_{X_2}f_2 \d\mu_2\right)^j = 0,
    \end{equation*}
    and, for all $n\in\N$,
    \begin{align*}
        \int_{X_{i,j}}f_{i,j}\cdot T_{i,j}^n f_{i,j}\d\mu_{i,j} &= \int_{X_{i,j}}\left(\prod_{a=1}^i f_1(x_a^{(1)})f_1(T_{i,j}^n x_a^{(1)})\right)\left(\prod_{b=1}^j f_2(x_b^{(2)})f_2(T_{i,j}^n x_b^{(2)})\right)\d\mu_{i,j}\\
        &= \left(\int_{X_1}f_1\cdot T_1^nf_1 \d\mu_1 \right)^i \left(\int_{X_2}f_2 \cdot T_2^n f_2 \d\mu_2 \right)^j.
    \end{align*}
    Furthermore, for all $n\in\N$ and $(i,j),(i',j')\in\supp(P)$ with $(i,j)\neq (i',j')$, since $X_{i,j}$ and $X_{i',j'}$ are independent in the product defining $X$, we have
    \begin{equation*}
        \int_{X}f_{i,j}\cdot T_{i',j'}^n f_{i',j'}\d\mu = \left(\int_{X_{i,j}}f_{i,j} \d\mu_{i,j} \right)\left(\int_{X_{i',j'}}f_{i',j'} \d\mu_{i',j'} \right) = 0.
    \end{equation*} 
    We therefore deduce
        \begin{equation*}
        \int_X f\d\mu = \sum_{(i,j)\in\supp(P)}a_{i,j}^{1/2}\left(\int_{X_1}f_1 \d\mu_1\right)^i\left(\int_{X_2}f_2 \d\mu_2\right)^j = 0
    \end{equation*}
    and
    \begin{align*}
    \int_X f\cdot T^n f \d\mu &= \sum_{(i,j)\in\supp(P)}a_{i,j}\int_{X_{i,j}}f_{i,j}\cdot T_{i,j}^nf_{i,j} \d\mu_{i,j}\\
    &=  P\left(\int_{X_1}f_1\cdot T_1^n f_1\d\mu_1, \int_{X_2}f_2\cdot T_2^n f_2 \d\mu_2\right).
    \end{align*}
    To conclude, recall that if $n\in R_1\cup R_2$, then
    \begin{equation*}
        \min_{i\in\{1,2\}}\left(\int_{X_i}f_i\cdot T_i^n f_i \d\mu_i\right) \leqslant \max_{i\in\{1,2\}}\left(\sup_{m\in R_i}\int_{X_i} f_i\cdot T_i^m f_i \d\mu_i\right)= -\eps,
    \end{equation*}
    and so the properties of $P$ supplied by Lemma \ref{lem2} imply that
    \begin{equation*}
        \int_X f\cdot T^nf \d\mu \leqslant -\delta <0.
    \end{equation*}
    Lemma \ref{lem1} therefore shows that $R_1\cup R_2$ is not a set of nice recurrence, as required.
\end{proof}

\end{document}